\documentclass{psp}

\usepackage{amssymb,latexsym, amsmath, amsxtra}
\usepackage{color}

\newcommand{\sym}{\mathrm{sym}}
\newcommand{\ext}{\mathrm{ext}}
\newcommand{\spin}{\mathrm{spin}}
\newcommand{\tr}{\mathrm{tr}}

\newcommand{\A}{{\mathbb A}}
\newcommand{\Q}{{\mathbb Q}}
\newcommand{\Z}{{\mathbb Z}}
\newcommand{\R}{{\mathbb R}}
\newcommand{\C}{{\mathbb C}}

\newcommand{\GL}{{\rm GL}}

\newcommand{\Sp}{{\rm Sp}}
\newcommand{\GSp}{{\rm GSp}}

\newcommand{\Lfunction}{\hbox{$L$-function}}

\renewcommand{\Hom}{{\rm Hom}}

\def\term#1{\textbf{\textit{#1}}}

\newcommand{\forget}[1]{}
\def\qdots{\mathinner{\mkern1mu\raise0pt\vbox{\kern7pt\hbox{.}}\mkern2mu
\raise3.4pt\hbox{.}\mkern2mu\raise7pt\hbox{.}\mkern1mu}}

\usepackage{enumitem}

\renewcommand{\labelenumi}{\roman{enumi})}

\newtheorem{lemma}{Lemma}[section]
\newtheorem{theorem}[lemma]{Theorem}
\newtheorem{intuition}{Intuition}[section]
\newtheorem{corollary}[lemma]{Corollary}
\newtheorem{proposition}[lemma]{Proposition}

\newtheorem{remark}[lemma]{Remark}
\newtheorem{conjecture}[lemma]{Conjecture}

\newcommand{\ncr}[1]{{\relax}}

\begin{document}

\title[Multiplicity one for analytic $L$-functions and applications]{Multiplicity one for analytic $L$-functions and applications}

\author[Farmer, Pitale, Ryan, and Schmidt]{DAVID W. FARMER\\
\texttt{Caltech 8-32, 1200 E California Blvd}\\
\texttt{Pasadena, CA 91125}\\
email: \texttt{farmer@aimath.org}\\
\nextauthor
AMEYA PITALE\\
\texttt{University of Oklahoma, Department of Mathematics, Norman, OK
73019-3103}\\
email: \texttt{apitale@math.ou.edu}\\
\nextauthor
NATHAN C. RYAN\\
\texttt{Bucknell University, Department of Mathematics, Lewisburg, PA 17837}\\
email: \texttt{nathan.ryan@bucknell.edu}\\
\nextauthor
RALF SCHMIDT \thanks{This work was supported by an AIM SQuaRE. Ameya Pitale and Ralf Schmidt were supported by National Science Foundation grant DMS  1100541.}\\
\texttt{University of North Texas, Department of Mathematics, Denton, TX
76203}\\
email: \texttt{ralf.schmidt@unt.edu}
}

\receivedline{Received \today}

\maketitle

\begin{abstract}
We give conditions for when two
Euler products are the same given that they satisfy a functional
equation and their coefficients
are not too large
and do not differ from each other by too much.
Additionally, we prove a number of multiplicity one type
results for the number-theoretic objects attached to
$L$-functions.  These results follow from our main result, which has
slightly weaker hypotheses than previous multiplicity one theorems for
$L$-functions.
Significantly stronger results are available when the L-function
is known to be automorphic.
\end{abstract}

\section{Introduction}
An analytic $L$-function is a Dirichlet series that converges absolutely
in some right half plane, has a meromorphic continuation to a
function of order~$1$ with finitely many poles, satisfies
a functional equation, and admits an Euler product. For example,
the (incomplete) $L$-functions attached to tempered, cuspidal automorphic
representations, or the Hasse-Weil  $L$-functions attached to
non-singular, projective, algebraic varieties defined over a number
field, conjecturally satisfy these conditions. 

In this paper, using standard techniques from analytic number theory,
we prove a strong multiplicity one result for such $L$-functions,
without reference to any underlying automorphic or geometric
object. We closely follow the work of Kaczorowski and Perelli \cite{KP} but redo their
arguments for two reasons.  First, our results are more
general in that they have slightly weaker hypotheses.  Second, we
think that the techniques should be better known, especially to those
who study $L$-functions automorphically.  
In the spirit of this second motivation, we offer a series of intuitions about the analytic properties of $L$-functions that allow us to prove our theorems.

The multiplicity one results we discuss in this paper are statements
which assert that if two $L$-functions are sufficiently close, then
they must be equal. A model 
example is:

\begin{theorem}\label{thm:SSMOintro}
 Suppose $L_1(s)=\sum a_nn^{-s}$ and $L_2(s)=\sum b_nn^{-s}$
 are Dirichlet series which continue to meromorphic functions of
 order $1$ satisfying appropriate functional equations and having
appropriate  Euler products.
 Assume that $L_1(s)$ and $L_2(s)$ satisfy the Ramanujan conjecture.
 Assume also that $a_p=b_p$ for almost all $p$. Then
 $L_1(s)=L_2(s)$.
\end{theorem}

The precise conditions on the functional equation and Euler product
are described in Section~\ref{sec:Ldefinition}. A  weaker version of Theorem~\ref{thm:SSMOintro}, requiring
equality of the local Euler factors instead of the $p$th Dirichlet
coefficients, is given in~\cite{RudSar}.
Theorem~\ref{thm:SSMOintro} is also a consequence of the main
result in~\cite{KP}.
The result we will actually prove, Theorem \ref{thm:SSMO}, is
stronger. First, instead of
requiring equality of the $p$th Dirichlet series coefficients, we
only require that they are close on average. Second, the Ramanujan
hypothesis can be slightly relaxed.

We will present some applications of strong multiplicity one for
$L$-functions. The first application is to cuspidal automorphic
representations of $\GL(n,\A_\Q)$, where $\A_\Q$ denotes the ring
of adeles of the number field $\Q$. Any such representation $\pi$
factors as $\pi=\otimes\pi_p$, where $\pi_p$ is an irreducible
admissible representation of $\GL(n,\Q_p)$ (we mean $\Q_p=\R$ for
$p=\infty$). Attached to $\pi$ is an automorphic $L$-function
$L(s,\pi)$, whose finite part is $L_{{\rm fin}}(s,\pi) = \prod_{p < \infty} L(s, \pi_p)$. The completion of $L_{{\rm fin}}(s,\pi)$ is known to be ``nice'', and
hence $L_{{\rm fin}}(s,\pi)$ is the kind of function to which Theorem \ref{thm:SSMOintro}
applies. At almost all primes $p$ we have
$L(s,\pi_p)=\det(1-A(\pi_p)p^{-s})^{-1}$, where $A(\pi_p)={\rm
diag}(\alpha_{1,p},\ldots,\alpha_{n,p})$ is 
a diagonal matrix whose entries are the Satake parameters at~$p$.  
The Ramanujan conjecture is the assertion that each $\pi_p$ is
tempered, which in this context implies that $|\alpha_{j,p}|=1$.  In
particular note that $L(s,\pi_p)$ is a polynomial in $p^{-s}$ and this
polynomial has all its roots on the unit circle.

An easy consequence of Theorem \ref{thm:SSMOintro} is the following.

\begin{theorem}\label{thm:SSMOglnintro}
 Suppose that $\pi$ and $\pi'$ are (unitary) cuspidal automorphic representations of $\GL(n,\A_\Q)$ satisfying $\tr(A(\pi_p))=\tr(A(\pi_p'))$ for almost all~$p$. Assume that both $L_{{\rm fin}}(s,\pi)$ and $L_{{\rm fin}}(s,\pi')$ satisfy the Ramanujan conjecture. Then $\pi=\pi'$.
\end{theorem}

Most statements of strong multiplicity one in the literature are
phrased in terms of $A(\pi_p)$ and $A(\pi_p')$ being conjugate,
instead of the much weaker condition of the equality of their traces.
Using the stronger version of Theorem \ref{thm:SSMOintro}, we will
in fact prove a stronger result which only requires that the traces
are close enough on average; see Corollary~\ref{thm:smogln} for the precise statement.

Our second application is to paramodular forms of degree $2$.
For such modular forms Weissauer \cite{W} has proved the Ramanujan
conjecture. The Dirichlet coefficients $a_p, b_p$ appearing in Theorem
\ref{thm:SSMOintro} are essentially the Hecke eigenvalues for the Hecke
operator $T(p)$. We therefore have the following:

\begin{theorem}\label{thm:m1siegelintro}
 Suppose $F_j$, for $j=1,2$, are Siegel paramodular eigenforms of weight
 $k_j$ for the paramodular group of level $N_j$, with Hecke eigenvalues $\mu_j(n)$. Suppose, too, that the $F_i$ are not Saito-Kurokawa lifts.  If
 $p^{3/2-k_1}\mu_1(p)=p^{3/2-k_2}\mu_2(p)$ for all but finitely many $p$, then $k_1 = k_2$, $N_1=N_2$ and  $F_1, F_2$ have the same eigenvalues for the Hecke operator $T(n)$ for all~$n$. Furthermore, $F_1$ is a scalar multiple of $F_2$.
\end{theorem}

The remarkable fact here is that the Hecke operator $T(p)$ alone
does not generate the local Hecke algebra at $p$. This Hecke algebra
is generated by $T(p)$ and $T(p^2)$. The fact that the coincidence
of the eigenvalues for $T(p)$ is enough is of course a global
phenomenon. Using the results of \cite{Sch}, we get the conclusion that $F_1$ is a multiple of $F_2$.
 Again, using the version in Theorem \ref{thm:SSMO},
we can prove a stronger result; see Theorem \ref{thm:m1siegel}.

We conclude the paper with a discussion of the relative
merits of the two types of multiplicity one theorems for
L-functions: those which only make use of the L-function
axioms, and those which make use of the underlying objects
associated to an L-function.

\ncr{We conclude the paper with a section in which discuss how our results could be applied to other kinds of $L$-functions. We do not prove any theorems in this final section but do outline three potential applications.  One application is making progress on a result that says for two $L$-functions associated to two automorphic representations $\pi_1,\pi_2$ of $\GL(n)$, it almost suffices to have equality between their coefficients at indices $p^2$, $p^3$ and $p^4$ to deduce equality between all their coefficients (there is a technical hypothesis we could not remove immediately).  The second potential application is to $L$-functions associated to the exceptional group $G_2$ and the final potential application is to degree 10 and 14 $L$-functions associated to Siegel modular forms.}

\subsection{Analytic \emph{vs.} automorphic $L$-functions}
The results we present in this paper primarily consider $L$-functions
defined axiomatically: Dirichlet series with a
functional equation and an Euler product, having an analytic
continuation that satisfies a functional equation.
We will refer to these as ``analytic'' $L$-functions.

All known analytic $L$-functions are (in some cases
conjecturally) the $L$-function of an arithmetic object.
The existence of such underlying objects may make it possible
to prove stronger multiplicity one results for the
$L$-functions.  For example, for $L$-functions of cuspidal
automorphic representations on $\GL(n)$, there are
results much stronger than what follows from our main result.
See
Section~\ref{sec:arithmetic}.  Combining this with Arthur's lift from $\Sp(4)$  to $\GL(4)$
would give a stronger version of the multiplicity one results
we have given for Siegel modular forms; see Theorem~\ref{thm:m1siegel}.

There are examples of $L$-functions which have not been proven to be
automorphic, such as certain symmetric power $L$-functions.
In such a case, the only option is to use
an analytic multiplicity one result, such as those we prove in this paper.

\subsection{Comparison with other analytic multiplicity one results}
We briefly
compare our result to other ``analytic'' strong multiplicity one
results due to Murty and Murty \cite{MM}, Murty \cite{M}, Kaczorowski
and Perelli
\cite{KP} and Kaczorowski \cite{Kac}.  Roughly speaking,
these earlier versions of strong multiplicity one place more and more
conditions on the two $L$-function.  In \cite{MM}, one can find the
hypothesis that $a_p=b_p$ and $a_{p^2}=b_{p^2}$, in \cite{M}
it is assumed that $a_p=b_p$ and some conditions on the twists
of the $L$-functions.  Later, in \cite{KP}, the conditions on the
twists of the $L$-functions is replaced with the assumption that the
$L$-functions have Euler products of a form similar to what we assume.

We highlight a few
notable differences between these earlier versions of strong
multiplicity one and ours, making a more thorough comparison between
the version by Kaczorowski \cite{Kac} as it is the most recent and is closest to our
version. First, in \cite{Kac} it is assumed that the $L$-functions
satisfy the Ramanujan bound while we only assume a partial Ramanujan
bound.  Second, \cite{Kac}  requires the coefficients to be close
in the sense $|a_p-b_p|\ll  p^{-\frac12 -\delta}$ for some
$\delta>0$, while we allow $\delta=0$.
In the proof, our result requires an analysis of possible
zeros on the $\sigma=\frac12$-line, while this is avoided if~$\delta>0$.
Finally, the main corollaries of our result do not follow if one
requires~$\delta>0$.  Despite these differences, the method of
proof is similar, being closely based on~\cite{KP}.

The most salient difference between this paper and the papers with the
earlier versions of strong multiplicity one is that we have
applications in mind.  The consequences we prove are of broad
interest and lie in different areas of number theory. We include the proofs of the analytic results for
completeness and to expose these techniques to as broad an audience as
possible.

\subsection*{Notation}

We review some notation from analytic number theory.
Given two functions $f(x),\,g(x)$, 
\begin{itemize}
\item we write $f(x)\sim g(x)$ as $x\to\infty$ if $\lim_{x\to\infty} f(x)/g(x)=1$;
\item we write $f(x) \ll g(x)$ as $x \to \infty$ if there exists $C>0$ and $x_0>0$ so that if $x > x_0$ then
$|f(x)| \leq C |g(x)|$; this is also written as
$f(x)=\mathcal{O}(g(x))$ as $x\to\infty$;
\item we write $f(x) = o(g(x))$ as $x\to\infty$ if $\lim_{x \to \infty} f(x)/g(x) = 0$.
\end{itemize}
In this paper we drop the phrase ``as $x\to\infty$'' when using 
the above notation.

\section{Multiplicity one for $L$-functions}\label{sec:SS}
In this section we describe the $L$-functions for which we will prove a
multiplicity one result.  As in other approaches to $L$-functions
viewed from a classical perspective, such as that initiated by Selberg~\cite{Sel},
we consider Dirichlet series with a functional equation and an Euler product.
However, in contrast to Selberg, we strive to make all our axioms
as specific as possible.  
Presumably (as conjectured by Selberg) these different axiomatic approaches
all describe the same objects:  $L(s,\pi)$ where $\pi$ is a cuspidal
automorphic representation of~$\GL(n)$.
For a detailed discussion of the L-functions we consider here, see~\cite{FPRS}.
An interesting alternative is the approach of Booker~\cite{Boo},
which describes $L$-functions in abstract terms, modeled on the
explicit formula.  

\subsection{$L$-function background}\label{sec:Ldefinition}

Before getting to $L$-functions, we recall
two bits of terminology that will be used in the following discussion.  
An entire function $f:\C\to\C$ is said to have \term{order at most} $\alpha$ if for all $\epsilon > 0$:
\[
f(s)=\mathcal{O}(\exp(|s|^{\alpha + \epsilon})).
\]
Moreover, we say $f$ has \term{order}
$\alpha$ if $f$ has order at most $\alpha$, and $f$ does not
have order at most $\gamma$ for any $\gamma<\alpha$.  The notion of
order is relevant because functions of finite order admit a factorization
as described by the Hadamard Factorization Theorem. The $\Gamma$-function and all $L$-functions have order~1.

In order to ease notation, we use the normalized $\Gamma$-functions
defined by:
\[
\Gamma_\R(s):=\pi^{-s/2}\,\Gamma(s/2)\ \ \ \ \text{ and }\ \ \ \ \Gamma_\C(s):=2(2\pi)^{-s}\,\Gamma(s).
\]

Below, we give the axioms defining an $L$-function as presented in \cite{FPRS}. Throughout the axioms, $s=\sigma+ it$ is a complex variable with $\sigma$
and $t$ real.
\begin{description}
\setlength{\itemindent}{-\leftmargin}
\setlength{\listparindent}{\parindent}
\item[Axiom $1$ (Analytic properties)]
$L(s)$ is given by a Dirichlet series
\index{Dirichlet series}
\[
\label{eqn:DS}\tag{Ax1.1}
L(s)=\sum_{n=1}^\infty \frac{a_n}{n^s},
\]
where $a_n\in \C$.

{\renewcommand{\labelenumi}{\alph{enumi})}
\begin{enumerate}

\smallskip

\item \emph{Convergence:} $L(s)$ converges absolutely for $\sigma>1$.

\smallskip

\item \emph{Analytic continuation:} $L(s)$ continues
to a meromorphic function of order $1$ having only finitely many poles,
with all poles 
lying on the $\sigma=1$ line.
\end{enumerate}
}

\medskip

\item[Axiom $2$ (Functional equation)]
\index{functional equation!axioms}
There is a positive integer $N$ called the \term{conductor} of the \Lfunction,
\index{conductor}
a positive integer $d$ called the \term{degree} of the \Lfunction,
\index{degree}
a pair of non-negative integers $(d_1, d_2)$ called the \term{signature} of the \Lfunction,
\index{signature}
where $d=d_1+2d_2$,
and complex numbers $\{\mu_j\}_{j=1}^{d_1}$ and $\{\nu_k\}_{k=1}^{d_2}$
called the \term{spectral parameters} of the \Lfunction,
\index{spectral parameters}
such that the \term{completed \Lfunction} 
\index{completed \Lfunction}\index{L-function@\Lfunction!completed}
\[
\label{eqn:Lambda}\tag{Ax2.1}
\Lambda(s) =\mathstrut  N^{s/2} 
\prod_{j=1}^{d_1} \Gamma_\R(s+ \mu_j)
\prod_{k=1}^{d_2} \Gamma_\C(s+ \nu_k)
\cdot L(s)
\]
has  the following properties:

{\renewcommand{\labelenumi}{\alph{enumi})}
\begin{enumerate}

\smallskip

\item \emph{Bounded in vertical strips:}
Away from the poles of the \Lfunction,
$\Lambda(s)$ is bounded in vertical strips $\sigma_1 \leq \sigma \leq \sigma_2$.

\smallskip

\item \emph{Functional equation:} There exists $\varepsilon\in \C$,
\index{functional equation!general}
called the \term{sign} of the functional equation, such that
\index{sign!of the functional equation}
\[
\label{eqn:FE}\tag{Ax2.2}
\Lambda(s) 
=\mathstrut  \varepsilon \overline{\Lambda}(1-s).
\]

%


\end{enumerate}
}

\medskip
\item[Axiom $3$ (Euler product)] There is a product formula
\index{Euler product}
\[
\label{eqn:EP}\tag{Ax3.1}
L(s)= \prod_{p \, {\rm prime}} F_p(p^{-s})^{-1},
\]
absolutely convergent for $\sigma > 1$.

{
\renewcommand{\labelenumi}{\alph{enumi})}
\begin{enumerate}

\smallskip

\item \emph{Polynomial:} $F_p$ is a polynomial with $F_p(0)=1$.

\item \emph{Degree:} Let $d_p$ be the degree of $F_p$.  If $p\nmid N$ then $d_p=d$,
and if $p\mid N$ then $d_p<d$. 

\smallskip
\end{enumerate}
}

\medskip
\item[Axiom $4$ (Temperedness)]   The spectral parameters and Satake parameters satisfy precise bounds.
{
\renewcommand{\labelenumi}{\alph{enumi})}
\begin{enumerate}

\smallskip

\item \emph{Selberg bound:} \label{axiom:preciseselbergeigenvalue}
For every $j$ we have
$\Re(\mu_j) \in \{0,1\}$ and
$\Re(\nu_k) \in \{\frac12,1,\frac32,2,...\}$.

\item \emph{Ramanujan bound:} Write $F_p$ in factored form as
\[
\label{eqn:ram-bound}\tag{Ax4.1}
 F_p(z) = (1-\alpha_{1,p} z)\cdots (1-\alpha_{d_p,p} z)
 \]
 with $\alpha_{j,p} \not = 0$. If $p\nmid N$ then $|\alpha_{j,p}| = 1$ for all $j$. If $p\mid N$ then $|\alpha_{j,p}|= p^{-m_j/2}$ for some $m_j\in \{0,1,2,...\}$, and  $\sum m_j \le d-d_p$.
\end{enumerate}
}
\smallskip

\item[Axiom $5$ (Central character)] There exists a Dirichlet character $\chi$ mod~$N$,
\index{central character!of an \Lfunction}
called the \term{central character} of the \Lfunction.
{
\renewcommand{\labelenumi}{\alph{enumi})}
\begin{enumerate}

\smallskip

\item \emph{Highest degree term:} For every prime $p$,
\[
\label{eqn:Fpchi}\tag{Ax5.1}
F_p(z)=1-a_p z + \cdots + (-1)^d\chi(p) z^d .
\]

\item \emph{Balanced:} We have
$\mathrm{Im} \left(\sum \mu_j +  \sum(2\nu_k+1)\right) =0$.

\item \emph{Parity:} The spectral parameters determine the
parity of the central character:
\[
\label{eqn:parityaxiom}\tag{Ax5.2}
\chi(-1) = (-1)^{\sum \mu_j +  \sum(2\nu_k+1)}.
\]
\end{enumerate}
}
\medskip
\end{description}

If $p\mid N$ then $p$ is called a \term{bad} prime, otherwise, $p$ is a \term{good} prime. 

The Ramanujan bound has been proven in very few cases, the most
prominent of which are holomorphic forms on $\GL(2)$ and
$\GSp(4)$.  See \cite{Sar} for a survey of what is known towards proving the Ramanujan bound. Also see \cite{BB}, \cite{Cl}, \cite{HT} and \cite{Li}. 

We write $|\alpha_{j,p}|\le p^\theta$, for some $\theta<\frac12$, to indicate progress toward the Ramanujan bound, referring to this as a ``partial Ramanujan bound.''

\subsection{A simple version of strong multiplicity one}

In the following proposition we use the term ``$L$-function'' in a precise sense,
referring to a Dirichlet series which satisfies a functional equation of
the form \eqref{eqn:Lambda}-\eqref{eqn:FE} 
with the restrictions
$\Re(\mu_j)\in\{0,1\}$ and $\Re(\nu_j)$ a positive integer or half-integer,
and having an Euler product satisfying \eqref{eqn:EP}.
We refer to the quadruple
$(d,N,(\mu_1,\ldots,\mu_J:\nu_1,\ldots,\nu_K),\varepsilon)$
as the \term{functional equation data} of the $L$-function.

\begin{proposition}\label{prop:smoL}
 Suppose that $L_j(s)=\prod_p F_{p,j}(p^{-s})^{-1}$, for $j=1,2$, are $L$-functions which satisfy a partial Ramanujan bound for
 some $\theta<\frac12$. If  $F_{p,1}=F_{p,2}$ for all but finitely
 many~$p$, then $F_{p,1}=F_{p,2}$ for all $p$, and $L_1$ and $L_2$
 have the same functional equation data.
\end{proposition}

In particular, the proposition shows that the functional equation data of an
$L$-function is well defined.  There are no ambiguities arising, say,
from the duplication formula of the $\Gamma$-function.
Also, we remark that the partial Ramanujan bound is essential.
One can easily construct counterexamples to the above proposition using Saito-Kurokawa lifts,
which have L-functions of the form $L(s,f)\zeta(s-\frac12)\zeta(s+\frac12)$
and do not satisfy a partial Ramanujan bound.

\begin{proof}
Let $\Lambda_j(s)$ be the completed $L$-function of $L_j(s)$ and
consider
\begin{align}\label{eqn:lambda}
\lambda(s)=\mathstrut&\frac{\Lambda_1(s)}{\Lambda_2(s)}\cr
=\mathstrut& \Bigl(\frac{N_1}{N_2}\Bigr)^{s/2}
\frac{\prod_{j} \Gamma_\R(s+ \mu_{j,1})
\prod_{k} \Gamma_\C(s+ \nu_{k,1})}
{\prod_{j} \Gamma_\R(s+ \mu_{j,2})
\prod_{k} \Gamma_\C(s+ \nu_{k,2})}
\prod_p \frac{F_{p,1}(p^{-s})^{-1}}{F_{p,2}(p^{-s})^{-1}}.
\end{align}
By the assumption on $F_{p,j}$, the
product over $p$ is really a finite product.
Thus, \eqref{eqn:lambda} is a valid expression for $\lambda(s)$ for
all~$s$.  

By the partial Ramanujan bound and the
assumptions on $\mu_j$ and $\nu_j$, we see that $\lambda(s)$ has
no zeros or poles in the half-plane $\Re(s)>\theta$.  But by the
functional equations for $L_1$ and $L_2$ we have
$\lambda(s) = (\varepsilon_1/\varepsilon_2)\overline{\lambda}(1-s)$.
Thus, $\lambda(s)$ also has no zeros or poles in the half-plane
$\Re(s) < 1-\theta$.  Since $\theta<\frac12$, we conclude that
$\lambda(s)$ has no zeros or poles in the entire complex plane.

If the product over $p$ in \eqref{eqn:lambda} were not empty, then it is a rational function in $p^{-s}$.
Hence, if it is not a constant, it has infinitely many zeros or poles of the form $s = \alpha + 2 \pi in/\log(p), n \in \Z$. 
The fact that $\{\log(p)\}$ for distinct $p$ is linearly independent over the
rationals implies that there are no cancellations, and hence $\lambda(s)$ has infinitely many zeros
or poles on some vertical line.
This is a contradiction, $F_{p,1}=F_{p,2}$ for all~$p$.

The $\Gamma$-factors must also cancel identically, because
the right-most pole of $\Gamma_\R(s+\mu)$ is at $-\mu$, and the right-most pole of $\Gamma_\C(s+\nu)$ is at $-\nu$.
This leaves possible remaining factors of the form 
$\Gamma_\C(s+1)/\Gamma_\R(s+1)$, but that also has poles
because the $\Gamma_\R$ factor cancels the first pole
of the $\Gamma_\C$ factor, but not the second pole.
Note that the restriction $\Re(\mu)\in\{0,1\}$ is a critical
ingredient in this argument.

This leaves the possibility that $\lambda(s)=(N_1/N_2)^{s/2}$,
but such a function cannot satisfy the functional
equation $\lambda(s) = (\varepsilon_1/\varepsilon_2)\overline{\lambda}(1-s)$
unless $N_1=N_2$ and $\varepsilon_1=\varepsilon_2$.
\end{proof}

Subsequent proofs of strong multiplicity one will all begin by considering
the ratio of completed L-functions:
$\lambda(s)={\Lambda_1(s)}/{\Lambda_2(s)}$.

\subsection{Lifts of an L-function}
It follows straight from the definition that the product of two L-functions
is an L-function.  But there are more interesting ways to generate
new L-functions from a given set of L-functions, and these will play
a role in the strong multiplicity one results we describe later.
This is reminiscent of the relationship between the Sato-Tate conjecture,
which is a statement about the distribution of coefficients of an L-function,
and analytic properties of the L-function's symmetric powers~\cite{HST, BGH}.

We will denote an L-function $L(s)$ by
$L(s, \mathbf{1})$ 
 and the ``lifts'' of $L$ by $L(s, \rho)$,
where the operation $\rho$ arises by repeatedly applying one of the three
basic operations we now describe. 
We only provide a description of the local factor of a lift
at a good prime.  It is a conjecture that the lift is in fact
an L-function, in which case (by Proposition \ref{prop:smoL})
the bad local factors are uniquely determined.

Let $S$ be the finite set of \emph{bad} primes $p$ of $L(s)$,
and write the good local factors of $L$ as
\begin{equation}
L^S(s) = \prod_{p \text{ good}} \prod_{j=1}^d (1-\alpha_{j,p} p^{-s})^{-1}.
\end{equation}
The partial symmetric and exterior power $L$-functions are defined as follows: 
\begin{equation}\label{partial-sym-defn}
L^S(s,\sym^n) =
\prod_{p \text{ good}}\:
\prod_{\substack{0\leq i_1,\ldots,i_d\leq n\\i_1+\ldots+i_d=n}} (1-\alpha_{1,p}^{i_1} \ldots \alpha_{d,p}^{i_d} p^{-s})^{-1}
\end{equation}
\begin{equation}\label{partial-ext-defn}
L^S(s,\ext^n) =
\prod_{p \text{ good}}\;
\prod_{1\leq i_1<\ldots<i_n\leq d} (1-\alpha_{i_1,p} \ldots \alpha_{i_n,p} p^{-s})^{-1}.
\end{equation}
Those L-functions have degree  $\binom{d+n-1}{n}$ 
and $\binom{d}{n}$, respectively.
Note that $L(s,\sym^1) = L(s, \ext^1) = L(s, \mathbf{1}) = L(s)$.

If $M(s)=\prod_p G_p(p^{-s})^{-1}$ is an $L$-function of degree $d_M$, with
\begin{equation}
G_p(z) = (1-\beta_{1,p} z)\cdots (1-\beta_{d_M,p} z),
\end{equation}
then we can define the partial Rankin-Selberg convolution
\begin{equation}\label{partial-R-S}
L^S\times M^S(s) = 
\prod_{p \text{ good}}\:
\prod_{j=1}^{d_L} \prod_{k=1}^{d_M}(1-\alpha_{j,p} \beta_{k,p} p^{-s})^{-1}.
\end{equation}
When $L(s) = L(s, \rho_1)$ and $M(s) = L(s, \rho_2)$, we denote
$L\times M(s)$ by $L(s, \rho_1 \times \rho_2)$.

\subsection{All lifts, sorted by tensor degree}
The symmetric square and exterior square of an L-function have different degrees,
yet their $p$th coefficients (for a good $p$) look similar:
$a_{p^2}$ and $a_p^2 - a_{p^2}$, respectively.  Both are weighted degree~2 polynomials
in the coefficients of the original L-function, if we assign weight~$j$
to~$a_{p^j}$.  

If $L(s,\mathbf{1})=\sum a_n n^{-s}$ is an L-function and $L(s, \rho) = \sum a(n, \rho) n^{-s}$ is a lift,
we will say that the \term{tensor degree} of $L(s, \rho)$ \term{with respect to}~$L(s, \mathbf{1})$
is~$k$, if at a good prime $p$ for $L(s,\mathbf{1})$,
the lifted coefficient $a(p, \rho)$ is a weighted degree~$k$ polynomial
in~$a_p, a_{p^2}, \ldots$.
The terminology comes from the fact that in the representation-theoretic
view of L-functions, $\rho$ can be found inside $V\otimes V \otimes \cdots \otimes V$,
with $k$ copies of $V$, where $V$ is the standard representation of
$\GL(d,\C)$.

Our main result, Theorem~\ref{thm:SSMOgeneral} will be expressed in terms
of lifts of a given tensor degree.
Since an L-function is determined by its coefficients at a good prime,
all of the tensor degree~$k$ lifts are products and quotients
of lifts of the form $L(s, \sym^{n_1} \times \cdots \times \sym^{n_r})$
where $1\le n_1\le \cdots \le n_r$ and $\sum n_j = k$.

There are
$p(k)$ such L-functions, where $p(k)$ is the number of
partitions of $k$.

\subsection{Intuition for the analytic properties of $L$-functions}\label{sec:intuitions}

In this section we describe 5 principles, or ``intuitions'',
which underlie the proof of our most general forms of strong multiplicity one
for axiomatically defined L-functions.

\begin{intuition}\label{intuition:factorpowerseries}  Power series are very flexible,
in the sense that they can be factored in many ways.
\end{intuition}

Consider this identity of formal power series, where $A_0=1$:
\begin{align}
\label{eq:iden1}
\sum_{n=0}^\infty A_n x^n =&
\left(1+A_1 x + B_2 x^2 + \sum_{j=3}^\infty B_j x^j\right)\cr
&\phantom{xx}\times\left(1+(A_2 - B_2) x^2 + \sum_{j=2}^\infty C_{2j}x^{2j}\right)\cr
&\phantom{xx}\times \left(1 + \sum_{j=3}^\infty D_j x^j\right),
\end{align}
where $B_j$ and $C_j$ can be chosen arbitrarily, and $D_j$ is a weighted homogeneous polynomial in
the $A_n$, $B_n$, and $C_n$ for $n\le j$. Here, the weights are given by the subscripts. For example $A_1A_2$ has weight $3$; so does $A_3$.

A generalization of \eqref{eq:iden1} is
\begin{align}
\label{eq:iden2}
\sum_{n=0}^\infty A_n x^n 
=& \prod_{n=1}^N \left(1+A_{n,n} x^n + \sum_{j=2}^\infty B_{n,nj}x^{n j}\right) \\
&\phantom{xx}\times \left(1+\sum_{j=N+1}^\infty A_{N+1,j}x^j \right),\nonumber
\end{align}
where $B_{m,k}$ can be chosen arbitrarily, and each $A_{m,k}$ is a 
weighted homogeneous polynomial
in $A_j$ and $B_{n,j}$ for $n < m$ and $j\le k$.
There are many ways to create such identities, because $f\in \C[[x]]$
is a unit provided $f(0)\not = 0$.

If one views the above as identities of formal power series,
then indeed the $B_j$ and $C_j$ in \eqref{eq:iden1},
or the $B_{m,k}$ in \eqref{eq:iden2}, can be chosen arbitrarily.
But for the applications in this paper, we need identities
of analytic functions.  This requires control on the growth of the
$A_{m,k}$, and typically this follows by choosing the $B_{m,k}$
to be polynomials whose degree and coefficients grow slowly
as functions of $m$ and~$k$.  If the series converge in a small
disc, then we have an identity for their analytic continuations.

Here is a simple application of the above intuition:

\begin{intuition}\label{intuition:ram-a_p} Assuming the Ramanujan bound,
the zeros and poles of $L(s)$ are primarily determined
by $a_p$, with $a_{p^2}$ playing a very minor role and
$a_{p^n}$ for $n\ge 3$ playing essentially no role at all.
\end{intuition}

In \eqref{eq:iden1} let $A_n = a_{p^n}$ and $x = p^{-s}$,
set $B_j = C_j = 0$, and write $D_{p^j}$ in place of $D_j$. 
We have
\begin{equation}\label{eqn:lfactor3}
\sum_{n=0}^\infty a_{p^n} p^{-n s} 
=
(1+a_p p^{-s})
(1+a_{p^2} p^{-2s})
(1+ \sum_{n=3}^\infty D_{p^n} p^{-n s}),
\end{equation}
where
\begin{equation}
D_{p^3} = a_{p^3} - a_p a_{p^2}
\ \ \ \ \ \ \ \ \ \text{and} \ \ \ \ \ \ \
D_{p^4} = a_{p^4} - a_p a_{p^3} + a_p^2 a_{p^2}
\end{equation}
and in general
$D_{p^n}$ is a weighted degree $n$ homogeneous polynomial in $a_{p^j}$ for
$j \le n$ where $a_{p^j}$ has weight $j$, and which has all
coefficients in $\{-1,0,1\}$.  Perhaps the easiest way to
understand the properties of $D_{p^n}$ is to rearrange~\eqref{eqn:lfactor3}
and expand the geometric series to get 
\begin{equation}
1+ \sum_{n=3}^\infty D_{p^n} p^{-n s}
= \left(\sum_{n=0}^\infty a_{p^n} p^{-n s}\right)
  \left(\sum_{n=0}^\infty (-1)^n a_p^n p^{-n s}\right)
  \left(\sum_{n=0}^\infty (-1)^n a_{p^2}^n p^{-2 n s}\right).
\end{equation}
Then expand the right side and equate coefficients.

Thus, for any L-function $L(s)$ we have
\begin{equation}
L(s) = 
{\mathcal L}_1(s)
{\mathcal L}_2(s)
{\mathcal L}_3(s), \label{L1L2L3}
\end{equation}
say, where the ${\mathcal L}_j(s)$ are Euler products whose local
factors are on the right side of \eqref{eqn:lfactor3}.  
If the coefficients satisfy the Ramanujan bound
$a_{p^n} \ll_\varepsilon  p^{\varepsilon n}$ then $D_{p^n}$
satisfies the same bound, so the Euler product
for ${\mathcal L}_j(s)$ converges absolutely for $\sigma > 1/j$.
Thus, zeros and poles of $L(s)$ in $\sigma \ge \frac12$ can only
come from (the analytic continuation of) ${\mathcal L}_1(s)$ and from the edge of the half-plane
of absolute convergence for ${\mathcal L}_2(s)$.  By the functional
equation, those determine all zeros and poles.

One might suspect that we have gained little,
because maybe all the zeros of $L(s)$ come
from ${\mathcal L}_2(s)$?  This is not the case, because of:

\begin{intuition}\label{intuition:euler-prod-zero}  Nice Euler products have finitely many zeros or
poles on their abscissa of convergence.
\end{intuition}

Here `nice' means that the Euler product has an analytic continuation
past the line of
convergence, and (suitably normalized) there is a  bound for the 
the $p$th coefficients.
An illustrative example is the Riemann zeta-function, which has one pole
on its abscissa of convergence.
See Lemma~\ref{lem:M2zerosimproved}.

\begin{intuition}  Bad Euler factors, and finitely many good Euler factors, can be ignored.
\end{intuition}

This is because, even assuming the weakest possible partial Ramanujan bound,
any finite set of local factors (good or bad)
is regular and nonvanishing in $\sigma>\tfrac12 - \delta$
for some $\delta > 0$.

The situation described in Intuition~\ref{intuition:ram-a_p} is more restrictive than we will
consider in this paper,
because we do not wish to assume a Ramanujan bound.
Instead of the Ramanujan bound, we use:

\begin{intuition}\label{intuition:higher} Analytic properties of the symmetric powers, exterior
powers, and other associated higher degree L-functions, can serve as a replacement for the
Ramanujan bound.
\end{intuition}

We illustrate with a slight modification of
the explanation of Intuition~\ref{intuition:ram-a_p}.
As described in Lemma~\ref{lem:coefficientidentities},
at a good prime the coefficient of $p^{-s}$ in $L(s,\sym^2)$ is $a_{p^2}$.
So similar to \eqref{eqn:lfactor3} except that the middle factor contains
terms for $p^{4s}$, $p^{6s}$, \ldots,  we have an analogue of
\eqref{L1L2L3} of the form 
\begin{equation}
L(s)=
\prod_{p\text{ good}} \left( 1 + a_p p^{-s} \right)
\cdot
L(2s,\sym^2)
\cdot
{\mathcal L}_3(s)\cdot h(s),
\end{equation}
where $h(s)$ is a product over the bad primes of $L(s)$.

By the properties of $D_{p^n}$ in the analogue of  \eqref{eqn:lfactor3}, 
one can check that if $L(s)$ satisfies a partial Ramanujan
bound $a_{p^j} \ll p^{j(\theta + \varepsilon)}$,
then $D_{p^n}$ satisfy the same bound.
Thus, the Euler product for ${\mathcal L}_3(s)$
converges absolutely when $3\sigma - 3\theta > 1$, so when $\sigma > \frac13 + \theta$.
The conclusion is that a partial Ramanujan bound with $\theta < \frac16$
implies that ${\mathcal L}_3(s)$ is regular and nonvanishing in a
neighborhood of~$\sigma\ge \frac12$. 
Thus, only the (analytic continuation of the) product involving $a_p$ and
the function $L(2s,\sym^2)$ can contribute
zeros or poles in that region, so by the functional equation only those
can contribute any zeros or poles to $L(s)$.

A similar identity one can obtain is
\begin{equation}
L(s)=
\prod_p \left( 1 + a_p p^{-s} + a_{p^2} p^{-2s} \right)
\cdot
L(2s,\ext^2)^{-1}
\cdot
{\mathcal L}_3(s) \cdot h(s),
\end{equation}
where again we use ${\mathcal L}_3(s)$ to represent an Euler
products which converges absolutely for $\sigma > \frac13 + \theta$.

There are similar
identities involving higher degree L-functions.
See Lemma~\ref{lem:factorL} for more examples.

The above intuition hides a subtle point.
Since $a_{p^2} \ll p^{2\theta +\varepsilon}$, the series for $L(s,\sym^2)$
is only known to
converge for $\sigma > 1 + 2\theta$.  Now, it is an axiom for the L-functions
we consider here, and a theorem for automorphic L-functions, that the
Dirichlet series converges absolutely for $\sigma > 1$, a fact that does not
follow from a partial Ramanujan bound.  
But, such an assumption for all symmetric powers would not be
a replacement for the Ramanujan bound, because it is equivalent to the
Ramanujan bound:

\begin{lemma} If $L(s)$ is an L-function in the sense of 
Section~\ref{sec:Ldefinition}, then $L(s)$ satisfies the
Ramanujan bound~$\theta=0$ if and only if the Dirichlet series
for every symmetric power $L(s, \sym^j)$ converges absolutely
for~$\sigma > 1$, if and only if no good local factor of any symmetric
power has a pole in $\sigma > 1$.
\end{lemma}

\begin{proof}
If $\alpha_p$ is one of the Satake parameters of $L(s)$
at a good prime~$p$, then $\alpha_p^j$ is one of the Satake parameters
of~$L(s, \sym^j)$.
\end{proof}

Since the Dirichlet series and Euler product for
$L(s, \sym^j)$ can only be assumed to converge absolutely for $\sigma > 1 + j\theta$,
what assumption can we make on $L(s, \sym^j)$ which is not just a disguised form of
the Ramanujan bound?  Combining Intuitions 2.1 and 2.5, which is made precise in Lemma 2.8,
we can write any L-function in the form
\begin{equation}
L(s) = (\text{less complicated terms})
    \cdot L(j s, \sym^j)
    \cdot h_{j+1}((j+1)s)
\end{equation}
where $h_{j+1}$ is a Dirichlet series whose $n$th coefficient is
bounded by $n^{(j+1)\theta}$.
We want all the terms to be nice for $\sigma \ge \frac12$.
That will require
$L(z, \sym^j)$ to be nice for $\Re z \ge \frac{j}{2}$.  The residual
term, $h_{j+1}(z)$ has the same convergence as:
\begin{equation}
\sum \frac{n^{(j+1)\theta}}{n^z} .
\end{equation}
That series converges for $\Re z > (j+1)\theta + 1$.
Thus, $h_{j+1}((j+1)s)$ is nice for $\sigma \ge \frac12$
provided $(j+1)/2 > (j+1)\theta + 1$, which is the same as
$\theta < \frac12 - \frac{1}{j+1}$. With this choice of $\theta$, we get $1+j\theta = \frac j2+\frac 1{j+1}$. Hence to get that $L(z, \sym^j)$ to be nice for $\Re z \ge \frac{j}{2}$, we will have to make the additional assumption that the analytic continuation of $L(z, \sym^j)$ extends from $\Re z > \frac j2+\frac 1{j+1}$ to $\Re z \geq \frac j2$.

Outside the context of our proof, for $j \ge 3$
we are not aware of any significance of the $\sigma = \frac{j}{2}$ line
for~$L(s, \sym^j)$.

\subsection{Measuring the distance between sequences}

Before stating our strong multiplicity one result, we describe various ways to measure that
two sequences on the primes are ``close together.''
This discussion is aimed at providing intuition to readers who are less
familiar with the analytic perspective on 
$L$-functions.

Suppose $\{a_n\}$ and $\{b_n\}$ are two sequences of complex numbers.
Let $A_n = a_n-b_n$ and 
consider the following statements, each of which is a different
way to say that the two functions are not too far apart
on the primes.
\begin{align}\label{apbpsqrt}
|A_p| \ll\mathstrut & 1/\sqrt{\mathstrut p}\phantom{\frac12}& \text{for all $p$}, \\
\label{eqn:apequal}
 \sum_{p\le X} p\,| A_p |^2 \log(p) \ll\mathstrut & X \phantom{\frac12} \\
\label{eqn:sigmato1}
\sum_p \frac{p\,|A_p|^2 \log p}{p^\sigma} \ll\mathstrut & \frac{1}{\sigma - 1} &
\text{ as } \ \sigma\to 1^+ \\
\label{eqn:liminf1}
\liminf_{\sigma \to 1^+} (\sigma - 1) \sum_p \frac{p \, |A_p|^2 \log p}{p^\sigma} < \mathstrut & \infty.
\end{align}

Note that the last two conditions include the assumption that the series
converges for~$\sigma > 1$.
All four conditions
convey the idea that $a_p$ and $b_p$ generally
don't differ
by more than $C/\sqrt{\mathstrut p}$.
In \eqref{apbpsqrt} that bound holds for all~$p$;
in \eqref{eqn:apequal} $a_p$ and $b_p$  can differ by a
fixed amount for finitely many~$p$;
and in \eqref{eqn:sigmato1} and \eqref{eqn:liminf1} they can differ by a
fixed amount for infinitely many~$p$, provided
the set of such $p$ is sufficiently thin.

Inequality~\eqref{apbpsqrt} implies  the other conditions.
For \eqref{eqn:apequal} this is immediate from the Prime Number Theorem,
\begin{equation}\label{PNT}
\sum_{p\le X} \log(p) \sim X ,
\end{equation}
and for \eqref{eqn:sigmato1} it follows from the fact that
the Riemann zeta-function has a simple pole at $s=1$. Likewise, \eqref{eqn:apequal} implies \eqref{eqn:sigmato1} by
partial summation.  Finally, \eqref{eqn:sigmato1} trivially implies \eqref{eqn:liminf1}.

Thus, we see a tension between writing a
condition which is easy to understand, and a condition which
allows more flexibility.  
We will use \eqref{eqn:liminf1} in the statement of our theorems,
but it can be helpful to realize that the results are still
true using a simpler looking condition
like \eqref{apbpsqrt} or \eqref{eqn:apequal}.

\subsection{Strong multiplicity one theorems for analytic $L$-functions}
In this section we state our main theorems.
\begin{theorem}\label{thm:SSMO}
 Suppose $L_1(s)$, $L_2(s)$ are Dirichlet series with Dirichlet
 coefficients $a_n$ and $b_n$ , respectively, which continue to
 meromorphic functions of order 1 satisfying functional equations
 of the form \eqref{eqn:Lambda}-\eqref{eqn:FE} with a partial Selberg bound $\Re(\mu_j),\ \Re(\nu_j)>-\frac12$ for both functions,
 and having Euler products satisfying
 \eqref{eqn:EP}. Assume
 the Dirichlet coefficients at the primes are close to each other
 in the sense that $A_p = a_p-b_p$ satisfies \eqref{eqn:liminf1}.

We have $L_1(s)=L_2(s)$ if any of the following conditions is satisfied:
 \begin{enumerate}
     \item \label{rama0} $L_1$ and $L_2$ satisfy the Ramanujan bound $\theta=0$.
  \item \label{ramaaverage} $L_1(s)$ and $L_2(s)$ satisfy a  partial Ramanujan bound with $\theta < \frac16$ and the quantity $B_p = p^{-\frac12}(a_{p^2}-b_{p^2})$ satisfies \eqref{eqn:liminf1}.
  \item \label{ramaaverage2} $L_1(s)$ and $L_2(s)$ satisfy a  partial Ramanujan bound with $\theta < \frac14$ and
both $B_p$ and $C_p = p^{-1}(a_{p^3}-b_{p^3})$ satisfy~\eqref{eqn:liminf1}.
 \end{enumerate}
\end{theorem}

We can replace bounds on the coefficients by conditions on the lifts
of the L-function.
Let $\Upsilon(L)$ denote the set of operations $\rho$ such that  the partial lifts $L(s, \rho)$  of $L$ 
have an analytic continuation to  a neighborhood of $\sigma \ge m/2$
with finitely many zeros or poles in that half-plane, where
$m$ is the tensor degree of the lift.

\begin{theorem}\label{thm:SSMOgeneral}
Suppose $L_1(s)$, $L_2(s)$ satisfy the conditions in the first paragraph of
Theorem~\ref{thm:SSMO}.
We have $L_1(s)=L_2(s)$ if for each $L_j$ there exists a positive integer $M_j$ such that
for $L_j$ 
the Ramanujan bound holds with 
\begin{equation}
\theta_j \le \frac{M_j -1}{2M_j + 2} = \frac12 - \frac{1}{M_j + 1} ,
\end{equation}
and in addition
all lifts $L_j(s, \rho)$ of tensor degree $ m \le M_j$
have $\rho \in \Upsilon(L)$.
\end{theorem}

Note that when $M_j=1$ the condition on $L_j$ is the Ramanujan bound $\theta=0$,
and there are no conditions on the lifts of~$L_j$.

Since the set of lifts of a certain tensor degree are generated by a proper subset, it is not necessary to require every lift of a given tensor degree to have nice analytic properties:
we collect some representative examples in the next theorem.

\begin{theorem}\label{thm:SSMOpartial}
Suppose $L_1(s)$, $L_2(s)$ satisfy the conditions in the first paragraph of
Theorem~\ref{thm:SSMO}.

We have $L_1(s)=L_2(s)$ if each $L_j$, independently, satisfies any
of the following conditions (perhaps the same one, perhaps different ones):
\begin{enumerate}
\item $\theta_j = 0$,
\item $\theta_j < \frac16$ and $\sym^2 \in \Upsilon(L_j)$, 
\item $\theta_j < \frac14$ and $\sym^2, \sym^3, 1\times \sym^2 \in \Upsilon(L_j)$.
\item $\theta_j < \frac{3}{10}$  and $\sym^2, \sym^3, 1\times \sym^2, \ext^2 \times \sym^2\in \Upsilon(L_j)$.
\end{enumerate}

We have $L_1(s)=L_2(s)$ if each $L_j$, independently, satisfies any
of the following conditions (perhaps the same one, perhaps different ones):
\begin{enumerate}[resume]
\item $\theta_j = 0$,
\item $\theta_j < \frac16$ and $\ext^2 \in \Upsilon(L_j)$,\label{sixthext2}
\item $\theta_j < \frac14$ and $\ext^2, \sym^3, 1\times \ext^2 \in \Upsilon(L_j)$. \label{quarterext2sym3}
\item $\theta_j < \frac14$ and $\ext^2, \ext^3, 1\times \sym^2 \in \Upsilon(L_j)$.
\item $\theta_j < \frac{3}{10}$  and $\ext^2, \sym^3, 1\times \ext^2, \ext^2 \times \sym^2\in \Upsilon(L_j)$.
\end{enumerate}
\end{theorem}

One can write down examples with $\theta_j < \frac13$ or $\theta_j < \frac{5}{14}$
just by assembling information presented in the proof, and examples for larger $\theta_j$
by doing more tedious algebra.

\subsection{Some technical lemmas}

In this section we provide the lemmas required for the proof of
Theorem~\ref{thm:SSMO}.
There are two types of lemmas we require.  
The first deals with manipulating Euler products and establishing zero-free
half-planes via the convergence of those products. The second deals with
possible zeros at the edge of the half-plane of convergence.

\subsubsection{Coefficients of lifts}

If $L(s)=\sum a_nn^{-s}$ is an L-function then for $\rho=\sym^n$ or $\ext^n$ 
we write the Dirichlet series for the lift as
\begin{equation}
L(s,\rho) = \sum_j a(j,\rho) \,j^{-s}.
\end{equation}
\begin{lemma}\label{lem:coefficientidentities}
 If $p$ is a good prime for $L$ then
 \begin{align}
\label{lem:coefficientidentitieseq1}a(p,\sym^n) =\mathstrut & a_{p^n}\\
\label{lem:coefficientidentitieseq2}a(p^2,\sym^2) =\mathstrut & a_{p^4}-a_pa_{p^3}+a_{p^2}^2\\
\label{lem:coefficientidentitieseq3}a(p^3,\sym^2) =\mathstrut & a_{p^6} - a_pa_{p^5} + a_p^2a_{p^4} + a_{p^3}^2-2 a_pa_{p^2}a_{p^3} + a_{p^2}^3\\
\label{lem:coefficientidentitieseq4}a(p^4,\sym^2) =\mathstrut & a_{p^8} - a_p a_{p^7} + a_p^2 a_{p^6} - a_{p^3} a_{p^5} + a_p a_{p^2} a_{p^5}  - a_p^3 a_{p^5}+ 2 a_{p^4}^2 - 3 a_p a_{p^3} a_{p^4} \nonumber\\
 &\phantom{xx}- a_{p^2}^2 a_{p^4} + 2 a_p^2 a_{p^2} a_{p^4} + 2 a_{p^2} a_{p^3}^2 + 
 a_p^2 a_{p^3}^2- 3 a_p a_{p^2}^2 a_{p^3} + a_{p^2}^4  \\
\label{lem:coefficientidentitieseq5}a(p^2,\sym^3) =\mathstrut & a_{p^6}-a_pa_{p^5}+a_{p^2}a_{p^4} \\
\label{lem:coefficientidentitieseq6}a(p,\ext^2) =\mathstrut & a_p^2 - a_{p^2}\\
\label{lem:coefficientidentitieseq7}a(p^2,\ext^2) =\mathstrut & a_p^4 - 3 a_p^2a_{p^2} + 2a_{p^2}^2 + a_pa_{p^3} - a_{p^4}\\
\label{lem:coefficientidentitieseq8}a(p^3,\ext^2) =\mathstrut & a_p^6 - 5 a_p^4a_{p^2} + 7a_p^2a_{p^2}^2 - 2a_{p^2}^3 + 3a_p^3a_{p^3}  - 6a_pa_{p^2}a_{p^3} + 2a_{p^3}^2 - 2a_p^2a_{p^4} \nonumber\\
& \phantom{x} + 2a_{p^2}a_{p^4} + a_pa_{p^5} - a_{p^6}\\
\label{lem:coefficientidentitieseq9}a(p,\ext^3) =\mathstrut & a_{p^3} + a_p^3 - 2 a_pa_{p^2}.
\end{align}
\end{lemma}

\begin{proof}
The coefficients of $L(s)$ are symmetric polynomials in the Satake parameters; in fact they are the completely homogeneous symmetric polynomials and, as such, form a generating set for all symmetric polynomials.  As the coefficients of the symmetric and exterior power L-functions are symmetric in the Satake parameters, we obtain the lemma by writing their coefficients in terms of this generating set.
Alternatively, one can translate the equalities into statements about finite-dimensional representations of $\GL(d,\C)$. For example, \eqref{lem:coefficientidentitieseq2} is equivalent to the statement
\begin{equation}\label{lem:coefficientidentitieseq10}
 \sym^2\circ\sym^2\oplus\sym^1\otimes\sym^3\cong\sym^4\oplus(\sym^2)^{\otimes2}.
\end{equation}
We leave the verification of the other identities as an exercise.
\end{proof}

\begin{lemma}\label{RankinSelberglemma}
 Suppose
 \begin{equation}
  L(s) = \sum_n \frac{a_n}{n^s}
  \ \ \ \ \ \ \ \ \ \
  \text{ and }
  \ \ \ \ \ \ \ \ \ \
  M(s) = \sum_n \frac{b_n}{n^s}
 \end{equation}
 are L-functions, with
 \begin{equation}
  L\times M(s) = \sum_n \frac{c(n, L \times M)}{n^s}.
 \end{equation}
 If $p$ is a good prime for both $L$ and $M$, then $p$ is a good prime for $L\times M$, and 
 \begin{align}
  \label{RankinSelberglemmaeq1}c(p, L \times M)=\mathstrut & a_p b_p\\
  \label{RankinSelberglemmaeq2}c(p^2, L \times M) =\mathstrut & a_p^2 b_p^2 - (a_p^2 b_{p^2} + a_{p^2}b_p^2) + 2 a_{p^2}b_{p^2}\\
  \label{RankinSelberglemmaeq3}c(p^3, L \times M) =\mathstrut & a_p^3 b_p^3 - 2(a_p^3  b_p b_{p^2} + a_p a_{p^2}b_p^3) 
                + ( a_{p}^3b_{p^3} + a_{p^3}b_p^3) + 5a_pa_{p^2}b_pb_{p^2}\cr
                & \phantom{x} -3(a_{p^3}b_pb_{p^2} + a_pa_{p^2}b_{p^3}) + 3a_{p^3}b_{p^3}.
 \end{align}
\end{lemma}
\begin{proof}
These identities are proved in a similar manner to Lemma \ref{lem:coefficientidentities}.
\end{proof}

\subsubsection{Manipulating $L$-functions}

As described in Section~\ref{sec:intuitions}, it is useful to be able to factor $L$-functions in a variety of ways.  In particular, different factorizations allow us to control where the zeros and poles might potentially come from and allow us to assume weaker versions of the Ramanujan bound in order to reach our conclusions.  In this section we present a variety of such factorizations.

\begin{lemma}\label{lem:factorL}
Suppose
\begin{equation}
\label{eqn:factorL}
L^S(s) = \prod_{p\ \mathrm{good}}\:\prod_{j=1}^d (1-\alpha_{j,p} p^{-s})^{-1},
\end{equation}
where $|\alpha_{j,p}|\le p^\theta$ for some $\theta\in \R$.
Then for $\sigma > 1+\theta$, 
\begin{align}\label{eqn:Lfactored}
L^S(s)=\mathstrut & \prod_{p \text{ good}} (1+a_p p^{-s})
                       \cdot h_2(2s),\cr
=\mathstrut & \prod_{p \text{ good}} (1+a_p p^{-s}) \cdot L^S(2s,\sym^2) \cdot h_3(3s),\cr
=\mathstrut & \prod_{p \text{ good}} (1+a_p p^{-s}) \cdot L^S(2s,\sym^2) \cdot L^S(3s,\sym^3) \cr
&\phantom{xxxxxxxx} \cdot L^S(3s,\mathbf{1} \times \sym^2)^{-1} \cdot h_4(4s)\cr
=\mathstrut & \prod_{p \text{ good}} (1+a_p p^{-s}) \cdot L^S(2s,\sym^2) \cdot L^S(3s,\sym^3) \cr
&\phantom{xxxxxxxx} \cdot L^S(3s,\mathbf{1} \times \sym^2)^{-1} \cdot L^S(4s, \ext^2 \times \sym^2) \cdot h_5(5s)\cr
=\mathstrut & \prod_{p \text{ good}} (1+a_p p^{-s} + a_p^2 p^{-2s} ) \cdot L^S(2s,\ext^2)^{-1} \cdot h_3(3s),\cr
=\mathstrut & \prod_{p \text{ good}} (1+a_p p^{-s} + a_p^2 p^{-2s} ) \cdot L^S(2s,\ext^2)^{-1} \cdot L^S(3s,\sym^3) \cr
&\phantom{xxxxxxxx} \cdot L^S(3s,\mathbf{1} \times \ext^2) \cdot h_4(4s),\cr
=\mathstrut & \prod_{p \text{ good}} (1+a_p p^{-s} + a_p^2 p^{-2s} ) \cdot L^S(2s,\ext^2)^{-1} \cdot L^S(3s,\ext^3) \cr
&\phantom{xxxxxxxx} \cdot L^S(3s,\mathbf{1} \times \sym^2)^{-1}\cdot h_4(4s),\cr
=\mathstrut & \prod_{p \text{ good}} (1+a_p p^{-s} + a_p^2 p^{-2s} ) \cdot L^S(2s,\ext^2)^{-1} \cdot L^S(3s,\sym^3) \cr
&\phantom{xxxxxxxx} \cdot L^S(3s,\mathbf{1} \times \ext^2)\cdot L^S(4s,\sym^2 \times \ext^2) \cdot h_5(5s),\cr
\end{align}
where  $h_j(s)$, which is different at each instance,
is regular and nonvanishing for $\sigma > 1 + j\theta$.
\end{lemma}

\begin{proof} Using \eqref{eq:iden1} and \eqref{eq:iden2}, or by direct
verification, we have
\begin{align}\label{eqn:sixfactors}
\sum &A_n x^n=(1 + A_1 x )\cr
    &\phantom{x}\times (1 + A_2 x^2 + (A_4-A_1A_3 + A_2^2)x^4 \cr
         &\phantom{xxxxx} + (A_6 -A_1 A_5 + A_1^2 A_4 + A_3^2 - 2 A_1 A_2 A_3 + A_2^3)x^6 + \mathcal O(x^8)) \cr
   &\phantom{x}\times (1+(A_3 - A_1 A_2) x^3 + \cr
         &\phantom{xxxxx} + (A_6 - A_1A_5 - A_2A_4 + A_1^2A_4+ A_1A_2A_3 - A_1^3A_3 - A_2^3 + A_1^2 A_2^2) x^6
\cr &\phantom{xxxxx}+ \mathcal O(x^9)) \cr
   &\phantom{x}\times (1+(A_1^2 A_2 - A_2^2) x^4 + \mathcal O(x^8)) \cr
   &\phantom{x}\times (1+(A_5 - A_1^3 A_2 + A_1 A_2^2 + A_1^2 A_3 - A_2 A_3  - A_1 A_4)x^5 + \mathcal O(x^{10}))\cr
   &\phantom{x}\times (1+(A_1^4A_2 - 2A_1^2A_2^2 + A_2^3 + A_1A_2A_3 - A_1^2A_4 - A_3^2 + A_2A_4 +A_1A_5 - A_6)x^6 \cr
            &\phantom{xxxxx} + \mathcal O(x^{12}))\cr
   &\phantom{x}\times (1+\mathcal O(x^7)).
\end{align}

Writing $A_n=a_{p^n}$ and using the 
identities in Lemma~\ref{lem:coefficientidentities} establishes the first
four lines of~\eqref{eqn:Lfactored}.
For the 3rd line in~\eqref{eqn:Lfactored}, one can recognize the tensor degree 3 lifts by
writing it as
\begin{equation}
\frac{1+ A_3 x^3 + (A_6 - A_1 A_5 + A_2 A_4) x^6 + \mathcal O(x^9)}{1 + A_1A_2 x^3 +
                       (2 A_2 A_4 - A_1^2 A_4 + A_1^3 A_3 - 2 A_1 A_2 A_3 + A_2^3)x^6 + \mathcal O(x^9)} .
\end{equation}
The other lines in~\eqref{eqn:Lfactored} are established similarly.
\end{proof}

We leave it as an exercise for the reader to use the information which has
already been presented to write down identities like~\eqref{eqn:Lfactored}
involving $h_6(s)$ and $h_7(s)$.

\subsubsection{Zeros at the edge of the half-plane of convergence}

The absolute convergence of an Euler product in a half-plane
$\sigma>\sigma_0$ implies that the function has no zeros or poles in that
region.  If the Euler product has a meromorphic continuation to a 
larger region, it could possibly have zeros or poles on the
$\sigma_0$-line.
The lemma in this section, which 
basically follows the proof of Lemma~1 of~\cite{KP},
says that if the Dirichlet coefficients $a(p)$ are small on average
then there are finitely many zeros or poles on the 
$\sigma_0$-line.
Our modification is that we only require the $L$-function
to satisfy a partial Ramanujan bound.

\begin{lemma}\label{lem:M2zerosimproved}
For $j=1,2,3$, suppose
\begin{equation}\label{Fj-defn}
F_j(s) := \sum_p \frac{p^{2-2j}|A(p^j)|^2\log(p)}{p^s}
\end{equation}
converges absolutely for $\sigma > 1$, and also there
exists $M_j > 0$ such that
\begin{equation}\label{eqns:apbounds1}
\liminf_{\sigma \to 1^+}\ (\sigma - 1) F_j(\sigma)  \le M_j .
\end{equation}
Then
\begin{equation}
L(s)=\prod_p (1 + A(p)p^{-s})(1 + A(p^2)p^{-2s})(1 + A(p^3)p^{-3s}) 
\end{equation}
is a nonvanishing analytic function in the half-plane~$\sigma>1$.
Furthermore, if $L(s)$ has a
meromorphic continuation to a neighborhood of $\sigma\ge 1$,
then $L(s)$ has at most $(M_1 + 2 M_2 + 3 M_3)^2$ zeros or poles on the
$\sigma=1$ line.
\end{lemma}

\begin{remark}\label{rem:M2M3theta}
Note that if $\theta < \frac23$ then \eqref{eqns:apbounds1} holds with $M_3=0$, and if
$\theta < \frac12$ then \eqref{eqns:apbounds1} holds with $M_2=0$.
\end{remark}

It will be clear from the proof that the above lemma generalizes to 
any number of $M_j$.  We wrote this version so that the proof looks
simpler, and because of the lack of obvious applications of a more general
version.

Corollary~\ref{cor:Mjzeros} is the analogue of the main lemma in~\cite{KP}.

\begin{corollary}\label{cor:Mjzeros}
Let
\begin{equation}
L(s)=\prod_p \sum_{j=0}^\infty A({p^j}) p^{-j s}
\end{equation}
and suppose $A(p^j)$ for $j=1$, $2$, $3$, satisfies the conditions in Lemma~\ref{lem:M2zerosimproved}
and for $j\ge 4$ we have $A(p^j) \ll p^{j\theta}$ for some $\theta < \frac34$.
Then $L(s)$ satisfies the conclusions of Lemma~\ref{lem:M2zerosimproved}.
\end{corollary}

Warning: Lemma~\ref{lem:M2zerosimproved} and Corollary~\ref{cor:Mjzeros} refer to
properties of $L(s)$ on $\sigma=1$.  The proofs of strong multiplicity one rely
on properties of $L(s)$ on $\sigma=\frac12$.  The change of variables
$s \to s - \frac12$ corresponds to $\theta \to \theta - \frac12$, so $\theta < \frac34$
in Corollary~\ref{cor:Mjzeros} corresponds to $\theta < \frac14$ in a
strong multiplicity one theorem.

\begin{proof}[Proof of Corollary~\ref{cor:Mjzeros}]
Note that
\begin{align}
\sum_{j=0}^\infty A({p^j})x^j =\mathstrut & (1 + A(p)x)(1 + A(p^2)x^2)(1 + A(p^3)x^3) (1 - A(p)A(p^2)x^3)(1 + \mathcal O(p^{4\theta}x^4))\cr
=\mathstrut & f_1(x) f_2(x) f_3(x) f_4(x) f_5(x),
\end{align}
say. Let $F_j(s) = \prod_p f_j(p^{-s})$.

By Lemma~\ref{lem:M2zerosimproved},
$F_4(s)=\prod_p \left(1-\tfrac{A(p)}{p^{-s}}\tfrac{A(p^2)}{p^{-2s}}\right )$
is a nonvanishing analytic function for $\sigma > 1$ because each of $\sum_p \tfrac{A(p)}{p^{-s}}$ and $\sum_p \tfrac{A(p^2)}{p^{-2s}}$ converge absolutely.  Also, $F_5(s)$ is a nonvanishing analytic function
for $\sigma > \frac14 + \theta$.  Since $\theta < \frac34$, $F_5(s)$ converges in a neighborhood of $\sigma \ge 1$.

The product $F_1(s) F_2(s) F_3(s)$ meets the conditions of Lemma~\ref{lem:M2zerosimproved},
so that completes the proof because  $F_4(s) F_5(s)$ is regular and nonvanishing in a neighborhood of $\sigma \ge 1$ and $L(s) = F_1(s) F_2(s) F_3(s) F_4(s) F_5(s)$.
\end{proof}

\begin{proof}[Proof of Lemma~\ref{lem:M2zerosimproved}] 
We first show that the logarithm of $L(s)$ converges absolutely for $\sigma > 1$,
which proves that $L(s)$ is analytic and nonvanishing in that region.
Since $\log(1 + Y) = Y + \mathcal O(Y^2)$, the logarithm of one of the factors of
$L(s)$ looks like
\begin{equation}\label{eqn:logAp}
\sum_p \Big(\frac{A(p^j)}{p^{j s}} + \mathcal O(A(p^j)^2 p^{-2 j \sigma})\Big).
\end{equation}
The absolute convergence of $F_j(s)$ in (\ref{Fj-defn}) implies (by ignoring the $\log(p)$ factor)
that the  big-$\mathcal O$ term above converges absolutely for $\sigma > 1-\frac{1}{2j}$,
so it defines an analytic function in a neighborhood of $\sigma \ge 1$.

For the main term in \eqref{eqn:logAp}, if $\sigma > 1$ then by Cauchy's inequality we have
\begin{equation}
\left(\sum_p \frac{|A(p^j)|}{p^{j \sigma}}\right)^2
\le \left(\sum_p \frac{|A(p^j)|^2}{p^{(2j-1)\sigma}}\right)
\left(\sum_p \frac{1}{p^{\sigma}}\right).
\end{equation}
The first sum on the right above converges absolutely by (\ref{Fj-defn}) and the condition $2j-2+\sigma \leq (2j-1)\sigma$ for $\sigma \geq 1$,
and trivially so does the second sum. 

Thus, $L(s)$ is the product of three nonvanishing analytic functions in $\sigma > 1$.

Now we consider the zeros of $L(s)$ on $\sigma=1$.
Taking the logarithmic derivative of $L(s)$ and using the
same argument as above for the lower order terms, we have
\begin{align}\label{eqn:ld2}
\frac{L'}{L}(s) =\mathstrut & h(s) - \sum_p \frac{A(p)\log(p)}{p^s} + 2\, \frac{A(p^2)\log(p)}{p^{2s}} 
 +3 \,\frac{A(p^3)\log(p)}{p^{3s}} ,
\end{align}
where $h(s)$ is regular in a neighborhood of $\sigma \ge 1$. 

Suppose $s_1,\ldots,s_J$ are zeros or poles of $L(s)$,
with $s_j = 1+i t_j$ having multiplicity $m_j$.
We have
\begin{equation}
\frac{L'}{L}(\sigma+i t_j) \sim \frac{m_j}{\sigma-1},
\ \ \ \ \ \ \ \
\mathrm{as}
\ \
\sigma \to 1^+,
\end{equation}
therefore 
\begin{equation}\label{eqn:klimit2}
- \sum_p
\left(
\frac{A(p)\log(p)}{p^{\sigma + it_j}}  
+ 2\,  \frac{A(p^2)\log(p)}{p^{2(\sigma + it_j)}}
+ 3\,  \frac{A(p^3)\log(p)}{p^{3(\sigma + it_j)}}
\right)
\sim \frac{m_j}{\sigma-1},
\ \ \ \ \ \ \ \
\mathrm{as}
\ \
\sigma \to 1^+.
\end{equation}
Now write
\begin{equation}\label{eqn:ksum}
k(s) = - \sum_{j=1}^J m_j \sum_p 
\left(
\frac{A(p)\log(p)}{p^{s+i t_j}} 
+2\, \frac{A(p^2)\log(p)}{p^{2(s+i t_j)}}
+3\, \frac{A(p^3)\log(p)}{p^{3(s+i t_j)}}
\right)
\end{equation}
By \eqref{eqn:klimit2} we have
\begin{equation}\label{eqn:klimitL2}
k(\sigma) \sim \frac{\sum_{j=1}^J m_j^2}{\sigma-1},
\ \ \ \ \ \ \ \
\mathrm{as}
\ \
\sigma \to 1^+.
\end{equation}
We will manipulate \eqref{eqn:ksum} so that
so that we can use \eqref{eqns:apbounds1} 
to give a bound on $\sum m_j^2$ in terms of $M_1$, $M_2$, and $M_3$.

By Cauchy's inequality and the assumptions on $M_\ell$ we have
\begin{align} \label{eqn:kbound2}
|k(\sigma)| \le \mathstrut &
\sum_{\ell=1}^3
\ell \left|
\sum_p  \frac{p^{(1-\ell)\sigma}A(p^\ell)\log(p)}{p^{\sigma}} 
\sum_{j=1}^J \frac{m_j}{p^{\ell it_j}}\right| \cr
\le \mathstrut &
\sum_{\ell=1}^3
\ell
 \left(\sum_p \frac{p^{2(1-\ell)\sigma}|A(p^\ell)|^2 \log(p)}{p^{\sigma}}\right)^{\frac12}
 \left( \sum_p  \frac{\log p}{p^\sigma} \biggl|\sum_{j=1}^J m_j p^{-\ell it_j} \biggr|^2\right)^{\frac12}\nonumber\\
\le\mathstrut  & (1+o(1)) 
\sum_{\ell=1}^3
\ell
\left(\frac{M_\ell^2}{\sigma-1} \right)^{\frac12}
\biggl(
\sum_{j=1}^J \sum_{k=1}^J  m_j m_k
\sum_p  \frac{\log p}{p^{\sigma + i\ell (t_j - t_k)} }
\biggr)^{\frac12} \nonumber\\
\sim & \frac{M_1 + 2 M_2 + 3 M_3}{(\sigma-1)^\frac12 } 
\left(
\sum_{j=1}^J \frac{m_j^2}{\sigma-1}
  \right)^{\frac12}
\ \ \ \ \ \ \ \
\mathrm{as}
\ \
\sigma \to 1^+.
\end{align}
In the last step we used the fact that the Riemann zeta function has a simple
pole at~1 and no other zeros or poles on the $1$-line.

Combining~\eqref{eqn:klimitL2} and \eqref{eqn:kbound2} we have 
$\displaystyle
\sum_{j=1}^J m_j^2 \le (M_1 +2  M_2 + 3 M_3)^2.
$
Since $m_j\ge 1$, we get $J \le (M_1 +2  M_2 + 3 M_3)^2$, and the proof is complete.
\end{proof}

\subsection{Proof of Theorem~\ref{thm:SSMO}}

The proof begins the same as that of Proposition~\ref{prop:smoL}, by considering the ratio of completed $L$-functions:
\begin{equation}\label{eqn:Lambdaratio}
\lambda(s) := \frac{\Lambda_1(s)}{\Lambda_2(s)},
\end{equation}
which is a meromorphic
function of order~1 and satisfies the functional equation $\lambda(s)=\varepsilon \overline{\lambda}(1-s)$,
where $\varepsilon = \varepsilon_1/\varepsilon_2$.
First we will show that $\lambda(s)$ has finitely many zeros or poles in $\sigma \ge \frac12$,
which by the functional equation implies that $\lambda(s)$ has finitely many zeros or poles 
in the complex plane.

By the assumptions on the $\Gamma$-factors, all zeros or poles of $\lambda(s)$
in $\sigma \ge \frac12$ must come
from $L_1(s)/L_2(s)$.  Consider a local factor of that ratio:
\begin{align}
&\frac{1+a_p x + a_{p^2} x^2 + a_{p^3} x^3 + \mathcal O(p^{4\theta} x^4)}
     {1+b_p x + b_{p^2} x^2 + b_{p^3} x^3 + \mathcal O(p^{4\theta} x^4)}\cr
&\phantom{xxxxx}=1+(a_p - b_p) x + (a_{p^2} - b_{p^2} + b_p(a_p - b_p))x^2\cr
&\phantom{xxxxxXX}+(a_{p^3} - b_{p^3} + (b_p^2 - b_{p^2}) (a_p - b_p) -b_p (a_{p^2} - b_{p^2}))x^3 + \mathcal O(p^{4\theta} x^4)\cr
&\phantom{xxxxx}= (1 + (a_p - b_p) x) (1 + (a_{p^2} - b_{p^2}) x^2) (1 + (a_{p^3} - b_{p^3}) x^3)\cr
&\phantom{xxxxxXX}\times (1 - b_p  (a_p - b_p) x^2) (1 - b_p (a_{p^2} - b_{p^2})) x^3) \cr
&\phantom{xxxxxXX}\times (1 - (2 b_p^2 - a_p b_p + a_{p^2} - 2 b_{p^2})(a_p - b_p) x^3) \cr
&\phantom{xxxxxXX}\times (1 + b_p (a_{p^2} - b_{p^2}) x^3) (1 + \mathcal O(p^{4\theta}x^4)))\cr
&\phantom{xxxxx}= (1 + (a_p - b_p) x) (1 + (a_{p^2} - b_{p^2}) x^2) (1 + (a_{p^3} - b_{p^3}) x^3)\cr
&\phantom{xxxxxXX}\times  (1 +  \mathcal O(p^\theta(a_p - b_p) x^2)) (1 + \mathcal O(p^{2\theta} (a_{p} - b_{p}) x^3)) \cr
&\phantom{xxxxxXX}\times (1 + \mathcal O(p^\theta (a_{p^2} - b_{p^2}) x^3)) (1 + \mathcal O(p^{4\theta}x^4)))\cr
&\phantom{xxxxx}= f_1(x) \times \cdots \times f_7(x),
\end{align}
say.  Let $F_j(s) = \prod_p f_j(p^{-s})$.

We will show that Lemma~\ref{lem:M2zerosimproved} applies to 
$F(s+ \frac12) := F_1(s+\frac12)F_2(s+\frac12)F_3(s+\frac12)$.
There are three cases to consider, but it suffices to consider the weakest one:
$\theta < \frac14$ and $B_p$ and $C_p$ satisfy~\eqref{eqn:liminf1}.
The assumptions on $A_p$, $B_p$, and $C_p$ involve the convergence of
sums with these summands:
\begin{equation}
\frac{p\, |a_{p}-b_{p}|^2 \log(p)}{p^\sigma},
\ \ \ \ \ \ 
\frac{|a_{p^2}-b_{p^2}|^2 \log(p)}{p^\sigma}
\ \ \ \ \ \
\text{and}
\ \ \ \ \ \
\frac{p^{-1} |a_{p^3}-b_{p^3}|^2 \log(p)}{p^\sigma}
\end{equation}
By replacing $a_n$ and $b_n$ by $a_n/\sqrt{n}$ and $b_n/\sqrt{n}$,
respectively, (equivalently, replacing $s$ by $s + \frac12$),
we see that condition \eqref{eqns:apbounds1} applies,
so $F(s+\frac12)$ has finitely many zeros or poles in $\sigma \ge 1$,
so $F(s)$ has finitely many zeros or poles in $\sigma \ge \frac12$. 

Now we check that $F_j(s)$, for $j=4,5,6,7$, defines a nonvanishing analytic function
in a neighborhood of~$\sigma \ge \frac12$.  For $F_7$, this requires $\theta < \frac14$.
For $F_4$, comparing to $F_1$ we only require $\theta < \frac12$.  The same condition is
sufficient for comparing $F_5$ to $F_1$, and for comparing $F_6$ to $F_2$.

Thus, we have established
that $\lambda(s)$ has finitely many zeros or poles in a neighborhood of
$\sigma \ge \frac12$, and so by the functional equation, finitely many
zeros or poles in the complex plane.

Since $\lambda(s)$ has order~$1$, 
by the Hadamard factorization
theorem there exists $A\in \C$ such that
\begin{equation}\label{eqn:exprat}
\lambda(s) = e^{A s} r(s)
\end{equation}
where $r(s)$ is a rational function.

By~\eqref{eqn:exprat}, as $|s|\to\infty$,
\begin{equation}\label{eqn:hadrat}
\lambda(s) = C_0 s^{m_0}  e^{A s} \bigl(1 + \mathcal O(s^{-1})\bigr),
\end{equation}
for some~$C_0\not=0$ and~$m_0\in \Z$.
On the other hand, if~$b(n_0)$ is the first non-zero Dirichlet coefficient (with $n_0>1$) of $L_1(s)/L_2(s)$,
then by \eqref{eqn:Lambdaratio} and Stirling's formula, as $\sigma\to\infty$,
\begin{equation}\label{eqn:stirdir}
\lambda(s) = \bigl(B_0 s^{B_1} e^{B_2 s\log s + B_3 s }(1 +\mathcal O(s^{-1}))\bigr)\bigl(1 + b(n_0) n_0^{-s}
+ \mathcal O((n_0+1)^{-\sigma}).
\end{equation}
Comparing those two asymptotic formulas, the leading terms must be equal, so
$B_0=C_0$, $B_1=m_0$, $B_2=0$, and $B_3=A$.

Now let $s\to\infty$ along the curve $s=\frac12 \log(t)/ \log(n_0)+it$.
On that curve, the final factor in \eqref{eqn:stirdir} equals
\begin{align}\label{eqn:sqrtmain}
1 + b(n_0) \frac{e^{-i t \log(n_0)}}{\sqrt{t}} &+
\mathcal O\left(\frac{1}{t^{\frac12 \frac{\log(n_0+1)}{\log(n_0)}}}\right) \cr
  = 1 +& b(n_0) \frac{e^{-i\frac{\pi}{4}-i t \log(n_0)}}{\sqrt{\mathstrut
s}}
+ \mathcal O\left(\frac{\log(|s|)}{s^{\frac32}}\right)
+ \mathcal O\left(\frac{1}{s^{\frac12 \frac{\log(n_0+1)}{\log(n_0)}}}\right).
\end{align}
We have assumed the positive integer $n_0$ exists, with $b(n_0) \not= 0$.
So there is a genuine main term of size $1/\sqrt{\mathstrut s}$
in~\eqref{eqn:sqrtmain},
which contradicts the big-$\mathcal O$ terms in \eqref{eqn:hadrat} and
\eqref{eqn:stirdir}.
Thus $b(n_0)$, the first non-zero Dirichlet coefficient of
$L_1(s)/L_2(s)$, does not exist, so we conclude $L_1 = L_2$.

An alternate way to complete the proof uses the fact that a Dirichlet series is an
almost periodic function, as in~\cite{KP}.

\subsection{Proof of Theorem~\ref{thm:SSMOpartial}}

We prove Theorem~\ref{thm:SSMOpartial}.
The proof of Theorem~\ref{thm:SSMOgeneral} is identical, because the
$p$th coefficients of the lifts of tensor degree~$m$ are
a basis for the symmetric polynomials of degree~$m$.

Suppose $L_1(s) = \sum a_n n^{-s}$ and $L_2(s) = \sum b_n n^{-s}$
satisfy the given conditions.  As in the previous proofs, the starting
point is the ratio of completed L-functions
\begin{equation}
\lambda(s) = \frac{\Lambda_1(s)}{\Lambda_2(s)} .
\end{equation}
And as previously, the main step is to show that $L_1(s)/L_2(s)$ contributes
finitely many zeros or poles in $\sigma\ge \frac12$.  

The theorem contains many cases, so we chose one for illustration.
Suppose $L_1$ satisfies \ref{sixthext2} and $L_2$ satisfies~\ref{quarterext2sym3}.
By Lemma~\ref{lem:factorL} we have
\begin{align}\label{eqn:L1overL2}
\frac{L^S_1(s)}{L^S_2(s)} =\mathstrut &
\frac{\prod_{p\ \text{good}} 
(1+a_p p^{-s} + a_p^2 p^{-2s}) }
{\prod_{p\ \text{good}} (1+b_p p^{-s} + b_p^2 p^{-2s})}  \cr
&\phantom{xxx}\times
\frac{ L_1^S(2s,\ext^2)^{-1} \cdot h_3(3s) }
{ L_2^S(2s,\ext^2)^{-1} \cdot   L^S_2(3s,\sym^3) \cdot L^S_2(3s, \mathbf{1}\times \ext^2)\cdot h_4(4s)}
\end{align}
where ``$p$ good'' refers to primes which are good for both $L_1$ and~$L_2$.
By the assumptions, the factor on the second line above is regular and
nonvanishing on $\sigma \ge \frac12$.

Writing the first factor in \eqref{eqn:L1overL2} as
\begin{equation}
\prod_{p\ \text{good}} \sum_{j=0}^\infty \mathcal{A}(p^j) p^{-js}
\end{equation}
we have $\mathcal{A}(1) = 1$,
\begin{align}\label{eqn:scriptA}
\mathcal{A}(p) =\mathstrut & a_p - b_p, \cr
\mathcal{A}(p^2) =\mathstrut & a_p (a_p - b_p) = a_p \mathcal{A}(p), \cr
\mathcal{A}(p^3) =\mathstrut & -(a_pb_p + b_p^2) (a_p - b_p) = -(a_pb_p + b_p^2) \mathcal{A}(p), \cr
\end{align}
and for all $j$, $\mathcal{A}(p^j) = \mathcal O(p^{j \theta^*})$ where
$\theta^* = \max(\theta_1, \theta_2) < \frac14$.
By assumption,
\begin{equation}\label{eqn:scriptAbound}
\liminf_{\sigma \to 1^+} (\sigma - 1) \sum_p \frac{p\, |\mathcal{A}(p)|^2 \log p}{p^\sigma}
\le \mathcal{M}_1
\end{equation}
for some $\mathcal{M}_1 < \infty$.  
By \eqref{eqn:scriptA}, the bounds on $a_j(p)$, and \eqref{eqn:scriptAbound},
we have
\begin{align}\label{eqn:scriptAboundHigher}
\liminf_{\sigma \to 1^+} (\sigma - 1) \sum_p \frac{p^{1-2\theta^*}\, |\mathcal{A}(p^2)|^2 \log p}{p^\sigma}
\le \mathstrut \mathcal{M}_1 , \cr
\liminf_{\sigma \to 1^+} (\sigma - 1) \sum_p \frac{p^{1-4\theta*}\, |\mathcal{A}(p^3)|^2 \log p}{p^\sigma}
\le \mathstrut \mathcal{M}_1 .
\end{align}

If we let $A(n) = \sqrt{n} \mathcal{A}(n)$, then \eqref{eqn:scriptAbound} is
exactly the same as the bound on the $j=1$ case of \eqref{eqns:apbounds1}
with $M_1 = \mathcal{M}_1$.  Also, \eqref{eqn:scriptAboundHigher} implies
the $j=2$ and $j=3$ cases of  \eqref{eqns:apbounds1} with
$M_2=M_3=0$.  Finally, $A(n)$ satisfies a partial Ramanujan bound
with $\theta = \frac12 + \theta^* < \frac34$.  Thus, we can apply Corollary~\ref{cor:Mjzeros}
to conclude that $\sum A(n) n^{-s}$ has finitely many zeros or poles
in $\sigma \ge 1$, which is equivalent to $L_1^S(s)/L_2^S(s) = \sum \mathcal{A}(n) n^{-s}$
having finitely many zeros or poles
in $\sigma \ge \frac12$.

That completes the main step in this proof.
The remainder of the proof is
exactly as in the proof of Theorem~\ref{thm:SSMO}.

\section{L-functions with an arithmetic source}\label{sec:arithmetic}

The $L$-functions we have been discussing are defined axiomatically.
All known and conjectured $L$-functions arise from an arithmetic
or automorphic object.  When such a connection is known, it is
possible to prove stronger multiplicity one results.

If $X$ is an algebraic variety then one can form its Hasse-Weil $L$-function~$L(s,X)$.
Serre's book~\cite{Ser}, in
Section~6.3 ``About $N_X(p)-N_Y(p)$,'' addresses the question
of what conditions on the Dirichlet coefficients imply that
the two varieties have the same $L$-function.  For example, suppose
$X$ and $Y$ are elliptic curves over $\Q$, with $L$-function coefficients $a(n,X)$ and $a(n,Y)$.
If $|a(p,X) - a(p,Y)|< 2-\varepsilon$ for all but finitely many~$p$,
then actually $L(s,X) = L(s,Y)$.  Thus, for elliptic curve $L$-functions one
can obtain the conclusion of Theorem~\ref{thm:SSMO} with significantly
weaker assumptions.

If $\pi=\otimes\pi_p$ is a cuspidal automorphic representations of the
group $\GL(n,\A_\Q)$, then $L(s, \pi)$ is a primitive degree-$n$ analytic $L$-function,
satisfying all the axioms of Section~\ref{sec:Ldefinition},
conjecturally for Axiom 4 (Temperedness).
That all axiomatically defined $L$-functions are automorphic (with suitably
restrictive axioms, of course), is likely; Selberg \cite{Sel} made that conjecture
for a particular set of axioms.

Automorphic $L$-functions, as opposed to those that are axiomatically defined, have much stronger multiplicity one theorems because
of progress on the \term{Selberg orthonormality conjecture}:

\begin{conjecture}[Selberg Orthonormality Conjecture]
 Suppose that $L_1$ and $L_2$ are primitive $L$-functions with Dirichlet
 coefficients $a(p)$ and $b(p)$. Then
 \begin{equation}\label{eqn:orthonormality}
  \sum_{p\le X} \frac{a(p)\overline{b(p)}}{p}=\delta({L_1, L_2}) \log\log(X)+O(1),
 \end{equation}
 where $\delta({L_1, L_2}) = 1$ if $L_1=L_2$, and $0$ otherwise.
\end{conjecture}

For the standard $L$-functions of cuspidal
automorphic representations on $\GL(n)$,
Rudnick and Sarnak~\cite{RudSar} proved Selberg's
orthonormality conjecture for $\pi\cong \pi'$ under the assumption of Hypothesis~H.

Rudnick and Sarnak's \term{Hypothesis H} is the assertion
$$
 \sum_p \frac{a(p^k)^2\log^2(p)}{p^k} < \infty
$$
for all $k\ge 2$.  For a given $k$, this follows from
a partial Ramanujan bound $\theta<\frac12 - \frac{1}{2k}$. 
Since $k\ge 2$,
Hypothesis~H follows from the partial Ramanujan bound
$\theta<\frac14$.
Rudnick and Sarnak~\cite{RudSar} proved Hypothesis~H for
$n=2$, $3$.
The case of $n=4$ for Hypothesis~H was proven by Kim~\cite{K}.  For
$\pi\not\cong \pi'$ the Selberg Orthonormality Conjecture was proven, unconditionally for $n\leq 4$ and under Hypothesis~H, independently by Avdispahi\'{c}-Smajlovi\'{c} \cite{AvSm} and
Liu-Wang-Ye \cite{LWY}.

As a consequence of those results, we have
\begin{corollary}\label{thm:smogln}
  Suppose that $\pi$, $\pi'$ are (unitary) cuspidal automorphic representations of $\GL(n,\A_F)$, and suppose
 \begin{equation}\label{eqn:smogln}
  \sum_{p\le X} \frac{1}{p}\left| \tr A(\pi_p) - \tr A(\pi_p')\right|^2 
  \le (2-\epsilon) \log\log(X)
 \end{equation}
 for some $\epsilon>0$ as $X\to\infty$.  If $n\le 4$, or if
Hypothesis~H holds for both $L_\mathrm{fin}(s,\pi)$ and $L_\mathrm{fin}(s,\pi')$
(in particular if the partial Ramanujan conjecture $\theta<\frac14$ is true
for $\pi$ and $\pi'$), then $\pi=\pi'$.
\end{corollary}

\begin{proof}  Since  $\pi$ and $\pi'$ are cuspidal automorphic
representations of $\GL(n,\A_F)$, the $L$-functions  $L_1(s) = L_\mathrm{fin}(s,\pi)$
and $L_2(s) = L_\mathrm{fin}(s, \pi')$ are primitive $L$-functions. Hence,
by~\eqref{eqn:orthonormality}
\begin{align}\label{eqn:apbp2X}
 \sum_{p\le X} \frac{1}{p} |a(p)-b(p)|^2
 =\mathstrut&\sum_{p\le X} \frac{1}{p}  \bigl( |a(p|^2 + |b(p)|^2 - 2 \Re(a(p)\overline{b(p)})\bigr) \cr
 =\mathstrut& 2\log\log(X) - 2 \delta_{L_1, L_2} \log\log(X)+ O(1) \cr
 =& \begin{cases}
   O(1) & \text{ if }  L_1= L_2 \cr
   2\log\log(X) + O(1) & \text{ if } L_1\not = L_2. \cr
  \end{cases}
\end{align}

We have $\sum_{p\le X} \frac{1}{p} |a(p)-b(p)|^2  \le (2-\epsilon)\log\log(X)$ for some $\epsilon > 0$. This implies that
$\epsilon \log\log(X)$ is unbounded, and hence (\ref{eqn:apbp2X}) implies $L_1(s) = L_2(s)$. This gives us $\pi = \pi'$. \end{proof}

Since the $p$th Dirichlet coefficient of $L(s,\pi)$ is $a_p = \tr A(\pi_p)$,
we see that condition \eqref{eqn:smogln} is similar to, but significantly
weaker than, conditions like \eqref{eqn:liminf1} which appear in the results
for axiomatically defined L-functions.  As a measure of how surprising and
strong Corollary~\ref{thm:smogln} is, we note that by the simple estimate
(which follows from the prime number theorem but predates it)
\begin{equation}\label{eqn:weakpnt}
\sum_{p\le X}\frac{1}{p} \sim \log\log(X),
\end{equation}
and the fact that $1.4^2 < 2$, condition \eqref{eqn:smogln} holds if
$|\tr A(\pi_p) - \tr A(\pi_p')| < 1.4$ for all but finitely many~$p$.
Not only that, but the conclusion $\pi=\pi'$ follows from only an
assumption on their \emph{traces}.

For $\GL(2,\A_\Q)$, the Ramanujan bound along with \eqref{eqn:weakpnt}
implies a version of a result of Ramakrishnan~\cite{Ram}:
if $ \tr A(\pi_p) = \tr A(\pi_p')$ for $\frac78+\varepsilon$ of all
primes~$p$, then $\pi=\pi'$. This result was extended by Rajan~\cite{Raj}.

Note that there is a hidden condition in Corollary~\ref{thm:smogln}:
the representations are cuspidal, and so the L-functions are primitive.
The conclusion of the theorem is false for non-primitive L-functions:
$L(s)$ and $L(s)L(s,\chi)$ have their $p$th coefficients differ by~$1$
at almost all~$p$.

\section{Applications}
\label{sec:applications}

By ``application'' we mean ``deduce that two objects have identical
invariants from their L-functions having similar coefficients''.

The theorems which provide a key step in applications
have the form ``If the L-functions of $X$ and $Y$ have
Dirichlet
coefficients $a_X(p)$ and $a_Y(p)$ which are close enough for
sufficiently many $p$, then $a_X(p) = a_Y(p)$ for all good~$p$.''
The two main examples in this paper are
Corollary~\ref{thm:smogln}
and
Theorem~\ref{thm:SSMO}.
Theorem~\ref{thm:smogln} is considerably stronger in the
sense that there are fewer restrictions on the relationship
between $a_X(p)$ and $a_Y(p)$, but the cost is that it
requires $X$ and $Y$ to arise from a cuspidal automorphic
representation on~$GL(n)$.

Here we make the point that Theorem~\ref{thm:SSMO}
is also useful, because there are L-functions which
satisfy the appropriate axioms, but which are not
known to be associated to a cuspical automorphic
representation.

There are two ways which Corollary~\ref{thm:smogln}
might not apply.  The first is that we do not yet
have sufficient knowledge about the automorphy of
lifts.  
For example, suppose $f$ and $g$ are $\GL(2)$ cusp
forms, and suppose their Fourier coefficients
$a_f(p^5)$ and $a_g(p^5)$ are close.  Since those
are the $p$th coefficients of the symmetric 5th
powers of their L-functions, one can use a multiplicity
one result to conclude that $a_f(p^5) = a_g(p^5)$
for all good~$p$.  If $f$ and $g$ are holomorphic
cusp forms, then, by \cite{NT1},  Corollary~\ref{thm:smogln} 
applies.  But if $f$ and $g$ are Maass forms,
then the required closeness of $a_f(p^5)$ and $a_g(p^5)$
is more restrictive because only Theorem~\ref{thm:SSMO}
is available

Continuing with this example of two $\GL(2)$ cusp forms, suppose $|a_f(p)|$ 
is close to $|a_g(p)|$ and we wish to conclude that
$|a_f(p)|=|a_g(p)|$ for all good~$p$.
Since $|a_f(p)|^2$ is the $p$th coefficient of
$L(s,f \times \overline{f})$ and similarly
for $|a_g(p)|^2$, we can apply Theorem~\ref{thm:SSMO}.
However, that L-function is not primitive, so
we cannot apply Corollary~\ref{thm:smogln} because
the associated automorphic representation is not
cuspidal.  (We acknowledge that this particular case
can be handled easily, but we want to emphasize that
primitivity is required to apply Corollary~\ref{thm:smogln}.)

Recent progress has shown that many L-function arise
from automorphic representations on $\GL(n)$, but
as long as there are analytic L-functions which have not
been shown to arise in that way, there is a place for
multiplicity one results which only invoke the axiomatic 
properties for L-functions.
For example, 
prior to the work of Arthur~\cite{Art1}, the spin L-function
of a Siegel modular form on $\Sp(4)$ was on that list.
The theorem below (which is a strengthening of
Theorem~\ref{thm:m1siegelintro}) can have weaker assumptions, and a much
simpler proof, thanks to Arthur's work, but for illustration
we give the version that uses Theorem~\ref{thm:SSMO}.

We refer the reader to Section 2 of \cite{Sch} for the following notations and definitions for Siegel modular forms
on the paramodular group~$K(N)$.
Let $k, N \in \Z_{>0}$. Let $S_k^{\text{new}}(K(N))$ be the space of holomorphic  Siegel modular new-forms of genus $2$ with weight $k$, with respect to the paramodular congruence subgroup of level $N$. For $F \in S_k^{\text{new}}(K(N))$, let $\mu_F(n)$ be the Hecke eigenvalue for the Hecke operator $T(n)$ for all integers $n$
coprime to~$N$.

\begin{theorem}\label{thm:m1siegel}
 For $i = 1,2$ let  $F_i \in S_{k_i}^{\text{new}}(K(N_i))$,  with Hecke eigenvalues $\mu_i(n), i = 1,2$ for all $n$ coprime to $N_i$. Assume $k_i \geq 2$, and the $F_i$s are not Saito-Kurokawa lifts. Let $N = {\rm lcm}(N_1, N_2)$. If
 \begin{equation}\label{eqn:m1siegel}
  \sum_{\substack{p\le X \\ p \nmid N}} p\,\log(p) \left|p^{3/2-k_1}\mu_1(p)-p^{3/2-k_2}\mu_2(p)\right|^2\ll X
 \end{equation}
 as $X\to\infty$, then $k_1=k_2, N_1 = N_2$ and $F_1$ is a scalar multiple of $F_2$.
\end{theorem}
\begin{proof}
For $i=1,2$ let $a_i(n)$ be the $n$th Dirichlet coefficient of the degree $4$ spin $L$-functions $L(s,F_i,\spin)$. For $p \nmid N$, if $\alpha_{i,p},\beta_{i,p}$ are the Satake $p$-parameters of $F_i$, then
\begin{align*}
a_i(p)&=\alpha_{i,p}+\alpha_{i,p}^{-1}+\beta_{i,p}+\beta_{i,p}^{-1}\\
a_i(p^2)&= \alpha_{i,p}^2+\alpha_{i,p}^{-2}+(\alpha_{i,p}+\alpha_{i,p}^{-1})(\beta_{i,p}+\beta_{i,p}^{-1})+\beta_{i,p}^2+\beta_{i,p}^{-2}+2.
\end{align*}
By Section 3.4 in \cite{Pit}, for $p \nmid N$,
\begin{equation}\label{mu-a-relation}
 \mu_i(p)= p^{k_i-3/2} a_i(p), \qquad \mu_i(p^2) = p^{2k_i-3}(a_i(p^2) - \frac 1p).
 \end{equation}
Hence, condition \eqref{eqn:m1siegel} translates into
\begin{equation}\label{eqn:m1siegelb}
  \sum_{\substack{p\le X \\ p \nmid N}} p\,\log(p) \left|a_1(p)-a_2(p)\right|^2\ll X.
\end{equation}
By Lemmas 2.3 and 2.5 of \cite{Sch}, both the $F_i$ are of general type {\bf (G)}. Hence, by Proposition 2.4 of \cite{Sch} and (\ref{eqn:m1siegelb}), their spin $L$-functions $L(s, F_i, \spin)$ satisfy the hypothesis of Theorem~\ref{thm:SSMO}. Since the $F_i$ are of type {\bf (G)}, it is known by \cite{W}, that the two $L$-functions satisfy the Ramanujan conjecture. Hence, Theorem~\ref{thm:SSMO} implies that $L(s, F_1, \spin)=L(s, F_2, \spin)$. This gives us $k_1 = k_2, N_1 = N_2 = N$ and $a_1(n) = a_2(n)$ for all $n$. In particular, (\ref{mu-a-relation}) implies that $\mu_1(p) = \mu_2(p)$ and $\mu_1(p^2) = \mu_2(p^2)$ for all primes $p$ coprime to $N$. Since $T(p)$ and $T(p^2)$ generate the $p$-component of the
Hecke algebra, it follows that $\mu_1(n)=\mu_2(n)$ for all $n$ coprime to $N$. Finally, using Theorem 2.6 of \cite{Sch}, we can conclude that $F_1$ is a scalar multiple of $F_2$.
\end{proof}

\proof[Acknowledgments]  We thank Farrell Brumley and Abhishek Saha for carefully reading an earlier version of this paper and for providing useful feedback on it.
We are also grateful for the careful reading done by the referee who, among other things, proposed stronger results and a better way to think about
applications of multiplicity one theorems, and helped us better organize the paper.

\end{document}